\documentclass[aap,preprint]{imsart}

\RequirePackage{amsthm,amsmath,amsfonts,amssymb}
\RequirePackage[numbers]{natbib}
\RequirePackage[colorlinks,citecolor=blue,urlcolor=blue]{hyperref}
\RequirePackage{graphicx}
\RequirePackage{mathtools,mathrsfs,epsfig,bm}
\RequirePackage[usenames,dvipsnames]{xcolor}
\RequirePackage[margin=2mm]{caption}
\RequirePackage{caption}

\startlocaldefs

\newtheorem{theorem}{Theorem}

\newtheorem{lemma}{Lemma}

\newtheorem{corollary}{Corollary}

\newtheorem{proposition}{Proposition}
\newtheorem{claim}{\bf Claim}
\theoremstyle{remark}
\newtheorem{definition}{Definition}
\newtheorem{remark}{Remark}

\def\N{\mathbb{N}}
\def\Z{\mathbb{Z}}
\def\Q{\mathbb{Q}}
\def\R{\mathbb{R}}

\def\P{\mathbb{P}}
\def\E{\mathbb{E}}
\def\FF{\mathscr{F}}
\def\GG{\mathscr{G}}

\def\NN{\mathcal N}

\renewcommand{\phi}{\varphi}
\renewcommand{\epsilon}{\varepsilon}
\def\emph{\bf\em}

\def\fG{\mathfrak G}
\def\fH{\mathfrak H} 

\def\o{$^\text{\rm o }$}
\def\red{\color{red}}

\def\black{\color{black}}

\definecolor{orange}{HTML}{e65c00}

\usepackage{soul}  

\def\edges{\mathsf E}
\def\piv{\text{piv}}
\def\hop{\text{hop}}
\def\short{\text{short}}
\def\hyper{\text{hyper}}
\def\internal{\text{int}}
\def\first{\text{first}}
\def\last{\text{last}}

\endlocaldefs

\begin{document}

\begin{frontmatter}

\title{ Probabilistic and analytical properties of the last passage 
percolation constant in  a weighted random directed graph}
\runtitle{Last passage  percolation  in weighted random directed graph}
\author{Sergey Foss, Takis Konstantopoulos \and Artem Pyatkin}


\begin{abstract}
To each edge $(i,j)$, $i<j$, of the complete directed graph on the integers 
we assign unit weight  with probability $p$
or weight $x$ with probability $1-p$, independently from edge to edge,
and give to each path weight equal to the sum of its edge weights.
If  $W^x_{0,n}$ is the maximum weight of all paths from $0$ to $n$
then $W^x_{0,n}/n \to C_p(x)$, as $n \to \infty$, almost surely,
where $C_p(x)$ is positive and deterministic.
We study $C_p(x)$ as a function of $x$, for fixed $0<p<1$, 
and show that it is a strictly increasing convex function 
that is not differentiable if and only if $x$ is a nonpositive rational
or a positive integer except $1$ or the reciprocal of it.
We allow $x$ to be any real number, even negative, or, possibly, $-\infty$.
The case $x=-\infty$ corresponds to the well-studied directed version 
of the Erd\H{o}s-R\'enyi random graph (known as Barak-Erd\H{o}s graph)
for which $C_p(-\infty) = \lim_{x \to -\infty} C_p(x)$ has been studied as
a function of $p$ in a number of papers.
\end{abstract}



\end{frontmatter}

\section{Introduction} 

The classical Erd\H{o}s-R\'enyi random graph model  on the set
of integers $\Z$ admits a straightforward directed version: 
a pair $(i,j)$ of vertices, $i<j$, is declared to be an edge directed from $i$ to $j$ with
probability $p$, independently from pair to pair. This was introduced by 
Barak and Erd\H{o}s \cite{BE84}.
Graph and order-theoretic properties of it were studied in \cite{BE84,ABBJ94,BOLBRI97}.
A quantity of interest for this graph is the behavior of the random variable
$L_n$ defined as the maximum length of all paths between
two vertices at distance at most $n$. Motivated by the theory of food webs
in mathematical ecology, Newman and Cohen \cite{NEWCOH86,NEWM92} showed that
$L_n/n$ converges in probability to a positive constant $C_p$,
as $n \to \infty$. It is known that $C_p$ is
a continuous function of $p$ and its properties in the light connectivity regime where
studied; in particular, the derivative at $0$ is equal to $e$. Using
different methods, we showed in \cite{FK03} that the limit above is in the almost
sure sense and obtained good bounds of the function $C_p$ on the whole interval 
$0\le p \le 1$, first by relating the graph to a Markov process (the so-called 
{\emph Infinite Bin Model}), and second by studying the convergence of it to stationarity.
In addition, functional central limit theorems were obtained in \cite{FK03}.
More recently, Mallein and Ramassamy \cite{MR16,MR19} studied a
more general infinite bin model and showed, in particular, 
that $C_p$ is an analytic function of $p$ for $p>0$.
They also obtained explicit estimates for  $C_p$ that improve those of  \cite{FK03}

Another area where versions of the {\emph Barak-Erd\H{o}s random directed graph} 
appear is the stochastic modeling of parallel processing systems.
When jobs arrive randomly in continuous time and cannot be processed independently
because of constraints between them then it is known that the stability of the resulting stochastic dynamical system is intimately related to the longest or heaviest path
in a random graph representing ordering preferences among jobs; see \cite{GEL86,ISONEW94}. In these systems it is often necessary to introduce weights
on the vertices of the graph as well.
In such a case, $L_n$ has to be modified to measure not length but total weight.
Letting then $W_n$ be the maximum weight of all paths between vertices at
distance at most $n$,  \cite{FMS14} shows that 
the growth of $W_n$ is very different depending on whether the second
moment of the typical edge weight is finite or not. The situation for
both edge and vertex weights was studied in \cite{FK18} and included the 
possibility that vertex weights be negative.
Yet another application of a continuous-vertex extension of  
Barak-Erd\H{o}s random directed graphs
 appear in the
physics literature: Itoh and Krapivsky \cite{IK12} introduce a version, called  
``continuum cascade model'' 
of the stochastic ordered graph with set of vertices in $[0,\infty)$ 
and study asymptotics for
the length of longest paths between $0$ and $t$, deriving recursive integral equations for its distribution.

When weights are introduced, the weighted Barak-Erd\H{o}s type of graphs
can be seen as long-range last-passage percolation models. These models
appear in physics and other areas and are typically defined by giving i.i.d.\
random weights to the points and/or edges of a lattice and asking for the behaviour
of the maximum weight path. Whereas in the Barak-Erd\H{o}s weighted or
unweighted graphs the underlying lattice is not {\em a priori} given,
it appears as a result of the analysis: there exists an bi-infinite  collection
of random vertices on which i.i.d.\ random weights appear. These
special random vertices have been called {\emph skeleton points} in
Denisov {\em et al.} \cite{DFK12} or as {\emph posts} in the
literature concerning order theoretic properties of unweighted Barak-Erd\H{o}s 
graphs \cite{ABBJ94,BOLBRI97}. We mention, in passing, that in
Denisov {\em et al.} \cite{DFK12} the edge probability $p$ was made
to depend on the endpoints of the edge also. In the same paper, the vertex set
was extended to be $\Z\times\{1,\ldots,K\}$ where $K$ is a finite
integer and was seen that the existence of posts resulted in a nontrivial 
CLT, identical to those obtained in last-passage percolation problems
and in longest-subsequence problems. Namely, the CLT associated to $L_n$
resulted in a nonnormal distribution but in one that governs the largest
eigenvalue of a certain $K\times K$ random matrix. More striking is 
the result when $K \to \infty$ \cite{KT13} where the use of skeletons of the directed
random graph on $\Z \times \Z$ was made to resemble the last-passage percolation
model of \cite{BM2005} with a corresponding CLT yielding a Tracy-Widom 
distribution. 

The constant $C_p$ has so far only been studied as a function of the edge probability
$p$. It is natural therefore to ask how it depends on the weight distribution
in a weighted Barak-Erd\H{o}s graph. In this paper, we study the following 
weight distribution: for two integers $i<j$, we assign weight $1$ to the edge $(i,j)$ with 
probability $p$ or weight $x$ with probability $1-p$, independently from edge to
edge. We call $G_p(x)$ this random weighted directed graph, allowing for $x$
to range in $[-\infty, +\infty]$.  The weight of a path is the sum of the weights
of its edges. If we let $\Pi_{i,j}$ be the collection of all finite increasing sequences
$(i_0, i_1, \ldots, i_\ell)$ such that $i_0=i$, $i_\ell=j$ then every element of 
the nonrandom set $\Pi_{i,j}$ is a path in $G_p(x)$.
However, in the Barak-Erd\H{o}s graph, the set of paths from $i$ to $j$ is random.
We can unify the two pictures by letting $G_p(-\infty)$ denote the Barak-Erd\H{o}s graph
declaring that a path with weight $-\infty$ is not a feasible path.

We first observe that the growth rate of the heaviest path, denoted by $C_p(x)$,
exists almost surely and in $L^1$, when $x>-\infty$
and that it is a strictly positive constant.
When $x=-\infty$, the Barak-Erd\H{o}s graph, $C_p(-\infty)$ is the same
as the constant $C_p$ studied in \cite{FK03,MR16,MR19} and also
in  \cite{NEWCOH86,NEWM92} (with a different notation).
In this case, the growth rate of the heaviest=longest path is in the almost sure
but not in the $L^1$ sense. We have that $\lim_{x \to -\infty} C_p(x) =
C_p(-\infty)$.
We then study the behaviour of $C_p(x)$ when $-\infty\le x \le \infty$.
We show that it is a convex function on $[-\infty, \infty]$ and then study
its smoothness properties showing that it fails to be differentiable 
when $x$ is a nonpositive rational number, a positive integer other than $1$,
or the inverse of such a positive integer.
We use entirely probabilistic-combinatorial methods based on the use of
the aforementioned skeleton points and on the construction of paths that 
have certain criticality properties.
We finally discuss extension of the model $G_p(x)$ and pose some
intriguing open problems.

Problems of differentiability of percolation constants arise in models in statistical
physics and have appeared previously in the context of first-passage percolation.
Steele and Zhang \cite{SZ}, motivated by an old problem due to Hammerlsey and Welsh,
studied differentiability of the first passage percolation constant 
$\mu(F)$ on the integer lattice $\Z^2$ when edge weights are i.i.d.\ nonnegative
random variables with common distribution $F$ (typically, $F$ is
Bernoulli distribution $p\delta_1 + (1-p) \delta_0$). 
Specifically, if $a_{0,n}$ denotes the infimum of the weights of all
paths between $(0,0)$ and $(n,0)$, then
$\mu(F) = \lim_{n\to \infty} a_{0,n}/n$, a.s.\ and in $L^1$, by Kingman's
subbaditive ergodic theorem. Replacing $F$ by $F_t$, the distribution 
of the edge weight shifted by $t$, Steele and Zhange proved that
the concave function $t \mapsto \mu(F_t)$ fails to be differentiable at $t=0$
when $p$ is in some left neighborhood of $1/2$.
Although this model is different from ours, there is a certain similarity in that
both papers use a combination of the probabilistic
techniques with a sample-path analysis of a class of deterministic graphs and
produce examples of paths having desired properties.
Very recently, Krishnan, Rassoul-Agha and Sepp\"al\"ainen \cite{KRS}
extended the result of \cite{SZ} by studying a higher-dimensional models,
more general weight distributions and first passage percolation constants along 
general directions and showed, in particular, nondifferentiability of $\mu(F_t)$
at points other than $0$. The main theorem of our paper, see Theorem \ref{thmmain} below,
identifies completely all the points of nondifferentiability of the last-passage percolation
constant $x \mapsto C_p(x)$. At the end of the paper we make some remarks
concerning how one could generalize the methods and results to more general
edge-weight distributions.

\section{The model and some basic properties}
\label{secbasic}
Let $p$ be strictly between $0$ and $1$ in order to avoid trivialities. 
We construct the family of random {\emph weighted directed graphs} $G_p(x)$, 
$-\infty \le x \le \infty$,
by first letting $\{\alpha_{i,j}:\, i,j \in \Z,\, i<j\}$ be an i.i.d.\ collection
of random variables on some probability space $(\Omega, \FF, \P)$ such that
\[
\P(\alpha_{i,j}=1)=p, \quad \P(\alpha_{i,j}=0)=q=1-p,
\]
and then letting the weight of $(i,j)$ be
\begin{equation}
\label{xweight}
w^x_{i,j} := \alpha_{i,j} + x (1-\alpha_{i,j}), \quad x > -\infty.
\end{equation}
If $x=-\infty$, then $w^{-\infty}_{i,j} := \lim_{x \to -\infty} w^x_{i,j}$,
which is $1$ if $\alpha_{i,j}=1$ or $-\infty$ otherwise.
If we agree that $-\infty$ denotes the absence of an edge, then,
for the special case when $x=-\infty$, we interpret $G_p(-\infty)$ as the 
the Barak-Erd\H{o}s graph, an {\emph unweighted random directed graph.}
It is often more convenient to be thinking of the complete directed graph 
$K(\Z)$ with vertices
the set $\Z$ of integers and edges the set $\{(i,j) \in \Z \times \Z:\, i<j\}$.
Then the subgraph of $K(\Z)$ consisting of all edges $(i,j)$ with $\alpha_{i,j}=1$ 
is $G_p(-\infty)$. 
It is often convenient and more descriptive 
to call an edge $(i,j)$ of $K(\Z)$ {\emph blue}
if $\alpha_{i,j}=1$ or {\emph red} if $\alpha_{i,j}=0$.

A path $\pi$ in $K(\Z)$ is any finite increasing sequence $(i_0, \ldots, i_\ell)$ of
integers. The pairs $(i_{k-1}, i_k)$, $1\le k \le \ell$, are the edges of $\pi$.
The number $\ell$ of edges of $\pi$ is its length, also denoted by $|\pi|$.
Such a path is a path in $G_p(-\infty)$ if 
$\alpha_{i_{k-1},i_k}=1$ for all $1\le k \le \ell$.

We let 
\[
\Pi_{i,j}:=\{(i_0,\ldots,i_\ell):\, i=i_0<i_1< \cdots < i_\ell=j, \ell \in \N\}
\]
be the set of all paths in $K(\Z)$ from $i$ to $j$. 
(The cardinality of $\Pi_{i,j}$ is the same as the number of subsets
of $\{i+1, \ldots, j-1\}$, i.e., $2^{j-i-1}$.)
The weight of a path $\pi=(i_0,\ldots,i_\ell)$ of $G_p(x)$ is
\begin{equation}\label{wpath}
w^x(\pi) = w^x_{i_0,i_1}+\cdots + w^x_{i_{\ell-1},i_\ell}.
\end{equation}
This works even when $x=-\infty$.
Indeed, $w^{-\infty}(\pi)=\ell$ if
$\alpha_{i_{k-1},i_k}=1$ for all $1\le k \le \ell$;
otherwise,  $w^{-\infty}(\pi)=-\infty$.
We are interested in the quantity
\begin{equation}\label{Wxij}
W^x_{i,j} := \max\{w^x(\pi):\, \pi \in \Pi_{i,j}\}.
\end{equation}
Note that setting $x=-\infty$ in this equation we obtain
\[
W^{-\infty}_{i,j}=
\begin{cases}
-\infty, & \text{ if there is no path of $G_p(-\infty)$ from $i$ to $j$},
\\
\max |\pi|,
& \text{ otherwise},
\end{cases}
\]
where the last maximum is over all paths $\pi$ in  $G_p(-\infty)$ 
from $i$ to $j$.
Hence
$W^{-\infty}_{i,j}$
is the maximum length of all paths from $i$ to $j$ in $G_p(-\infty)$, provided
that such a path exists.
In other words, formula \eqref{Wxij}, appropriately interpreted,
gives the quantity of interest in all cases, that is, for all graphs $G_p(x)$,
$-\infty \le x \le \infty$.

\begin{theorem}[\cite{FK03}]
The following holds almost surely:
\[
C_p
:= \lim_{n \to \infty} \frac{(W^{-\infty}_{0,n})^+}{n}
= \inf_{n \ge 1} \frac{\E (W^{-\infty}_{0,n})^+}{n},
\]
where $C_p>0$. The first limit also holds in the $L^1$ sense.
\end{theorem}
As observed in \cite{FK03}, perhaps the quickest way to obtain this is to prove
that the quantity $L_{i,j}$, defined as the maximum length of all
paths in $G_p(-\infty)$ starting and ending at points between $i$ and $j$, satisfies
$L_{i,k} \le L_{i,j} + L_{j,k}+1$ for all $i < j < k$, and hence,
using Kingman's subbaditive ergodic theorem \cite[Theorem 10.22]{KALL02},
we have $\lim_{n \to \infty} L_{0,n}/n = C_p$, almost surely and in $L^1$, 
for some deterministic $C_p$,
and $C_p = \inf_{n \ge 1}\E L_{0,n}/n > 0$.
The quantity $(W_{0,n}^{-\infty})^+$, being the
maximum length of all paths  in $G_p(-\infty)$ from $0$ to $n$, is obviously
$\le L_{0,n}$. On the other hand, as shown in \cite{FK03}, 
we have that eventually the sequences $L_{0,n}$ and $(W_{0,n}^{-\infty})^+$
are equal almost surely.

\begin{theorem}
\label{thmx}
For $-\infty < x < \infty$, we have
\[
C_p(x) =  \lim_{n \to \infty} \frac{W^x_{0,n}}{n}
= \lim_{n \to \infty} \frac{(W^x_{0,n})^+}{n},
\]
almost surely and in $L^1$.
\end{theorem}
One way to obtain this result is to observe that we have superadditivity:
$W^x_{i,j} + W^x_{j,k} \le W^x_{i,k}$.
Again, by Kingman's theorem, the first limit exists almost surely and in $L^1$
and equals a deterministic constant $C_p(x)$. It is a positive constant because
$C_p(x) \ge C_p(-\infty) = C_p > 0$.

The main result of this paper is stated as follows.

\begin{theorem}
\label{thmmain}
The function $x \mapsto C_p(x)$ is differentiable everywhere except
when $x$ is a nonpositive rational or equal to $k$ or $1/k$ for
some integer $k \ge 2$.
\end{theorem}

This theorem will be proved in two steps. First, by proving that $C_p(x)$ is
not differentiable at $x$ if and only if $x$ is ``critical''
(in the sense of Definition \ref{defcrit} below), see Theorem \ref{thmdiffcrit}.
Second, by identifying the set of critical points, see Theorem \ref{thmcrit}.

\subsection{Skeleton points and a representation of the inter-skeleton structure}
To obtain further information about the constant $C_p(x)$ as a function of
$x$ we need the notion of skeleton points. We recall the notion below,
along with a fresh look at its structure.
We will say that a path is blue if all its edges are blue
(i.e., $\alpha_e=1$ for all edges $e$ of the path)
or red if all its edges are red
A blue path is a path in the graph $G_p(-\infty)$.
We say that $i$ is a {\emph skeleton point} \cite{FK03,DFK12,FK18}
(or post, in the terminology of \cite{ABBJ94,BOLBRI97}) 
if for all $j<i<k$ there is a blue path from $j$ to $k$
that contains $i$.
The random set of skeleton points is denoted by $\mathcal S$.
Clearly, $\mathcal S$ is stationary (i.e., it has a law that is invariant
under translations), 
it has infinitely many points almost surely, and 
the probability that a fixed integer $i$ is contained in $\mathcal S$ 
does not depend on $i$. This probability is the rate of $\mathcal S$ and is given by
\[
\gamma:= \P(i \in \mathcal S) = \prod_{k=1}^\infty (1-q^k)^2.
\]
We let $\Gamma_0$ be the largest skeleton point that is less than
or equal to $0$, and $\Gamma_1$ 
be the next skeleton point after $\Gamma_0$. We 
thus let $\mathcal S = \{\Gamma_k:\, k \in \Z\}$, where
\[
\cdots < \Gamma_{-1} < \Gamma_0 \le 0 < \Gamma_1 < \Gamma_2 < \cdots
\]
The constant $\gamma$ can be alternatively expressed as
\[
\gamma  = \frac{1}{\E(\Gamma_1-\Gamma_0|\Gamma_0=0)}
= \frac{1}{\E(\Gamma_k-\Gamma_{k-1}|\Gamma_0=0)},
\]
for all $k \in \Z$, thanks to stationarity.

When $u\le v$ are two integers, we write $[u,v]$ for the set of integers
$j$ such that $u \le j \le v$. We  also use the abbreviations
\[
G_p \equiv G_p(-\infty), \quad 
G_{p,u,v}
= \text{ the restriction of $G_p$ on the set of vertices $[u,v]$}.
\]

From previous work, we know that
\begin{lemma}[$G_p$  regenerates over $\mathcal S$, \cite{FK03,FK18}]
\label{indlem}
If $G_p^{(n)}=G_{p, \Gamma_{n-1},\Gamma_n}$ 
is the restriction of the graph $G_p$ on 
$[\Gamma_{n-1},\Gamma_n]$, then the marked point process with points at 
$\Gamma_n$ and marks $G_p^{(n)}$, $n \in \Z$, forms a stationary regenerative 
process. In particular, 
\\
(i) $\mathcal S$ is a stationary renewal process 
and 
\\
(ii) conditional on the event $\{\Gamma_0=0\}$ we have
that $G_p^{(n)}$, $n \in \Z$, is an i.i.d.\ sequence of finite random
directed graphs.
\end{lemma}
For a proof of this lemma, see \cite[Lemma 4]{DFK12}.
Let  
\[
\Delta:=\Gamma_2-\Gamma_1.
\]
Due to Lemma \ref{indlem}(i)
\[
\gamma = \frac{1}{\E \Delta} = \frac{1}{\E(\Gamma_{k+1}-\Gamma_k)},
\quad k \neq 0.
\]
Then 
\[
\P(\Delta=n) = \P(\Gamma_2-\Gamma_1=n)
=\P(\Gamma_1=n|\Gamma_0=0)
= \frac{\P(\Gamma_0=0, \Gamma_1=n)}{\P(\Gamma_0=0)}.
\]
Let $\{i \leadsto j\}$ denote the event that there is a
blue path from $i$ to $j$.
Define
\[
A_{u,v} := \bigcap_{u\le j< v} \{ j \leadsto v\}, \quad
B_{u,v} := \bigcap_{u < j \le v} \{u \leadsto j\}
\]
\[
A_i := \bigcap_{j<i} \{ j \leadsto i\},\quad
B_i := \bigcap_{j>i} \{ i \leadsto j\}.
\]
Then
\[
\{i \in \mathcal S \} = A_i \cap B_i.
\]
Define also
\[
F_{0,n} := \bigcap_{j=1}^{n-1} \bigg(
\{i \not \leadsto j \text{ for some } 0<i<j\}
\cup
\{j \not \leadsto i \text{ for some } j<i<n\} \bigg).
\]
Then
\begin{align}
\{\Gamma_0=0, \Gamma_1=n\} 
&= \{0 \in \mathcal S, \, 1 \not \in \mathcal S, \ldots, n-1 \not\in\mathcal S,
n \in \mathcal S\} \nonumber
\\
&= A_0 \cap B_0 \cap  (A_1 \cap B_1)^c 
\cap \cdots \cap (A_{n-1} \cap B_{n-1})^c  \cap A_n \cap  B_n
\nonumber
\\
&= A_0 \cap B_{0,n} \cap  F_{0,n}\cap A_{0,n} \cap  B_n.
\label{elementary}
\end{align}
The reason for this equality is elementary. If we let
$F:=\{1 \not \in \mathcal S, \ldots, n-1 \not\in\mathcal S\}$ then
$\{0 \in \mathcal S\} \cap F \cap \{n \in \mathcal S\} 
= \{0 \in \mathcal S\} \cap F_{0,n} \cap \{n \in \mathcal S\}$
because if we know that $0$ and $n$ are skeleton points then the event
$F$ that for some point $1\le j \le n-1$ fails to be reachable from a lower
point or fails to reach a higher point is necessarily equivalent to
$F_{0,n}$. Thus, 
$\{0 \in \mathcal S, \, 1 \not \in \mathcal S, \ldots, n-1 \not\in\mathcal S,
n \in \mathcal S\} = A_0 \cap B_0 \cap F_{0,n} \cap A_n \cap B_n$.
Furthermore, $B_0 \cap B_n = B_{0,n} \cap B_n$ and
$A_0 \cap A_n = A_{0,n} \cap A_n$. This proves \eqref{elementary}.
It is convenient to group together the middle three events on
the right hand side of  \eqref{elementary}
and let
\[
H_{0,n} := B_{0,n} \cap A_{0,n} \cap F_{0,n},
\]
so that $\{\Gamma_0=0, \Gamma_1=n\} = A_0 \cap B_n \cap H_{0,n}$.
Since $A_0, H_{0,n}, B_n$ are independent we have
\[
\P(\Gamma_0=0, \Gamma_1=n) = \P(A_0) \P(B_n) \P(H_{0,n}).
\]
On the other hand,
\[
\P(\Gamma_0=0) = \P(A_0 \cap B_0) = \P(A_0) \P(B_0),
\]
and since $\P(B_n) = \P(B_0)$ we have obtained that
\begin{proposition}
\label{deltaH}
\[
\P(\Delta=n) = \P(H_{0,n}),
\]
where 
$H_{0,n}$ is the event that for any vertex $j$ between $0$ and $n$
there is a blue path from $0$ to $n$ containing $j$ and there is a vertex $i\neq j$
such that there is no blue path from $\min(i,j)$ to $\max(i,j)$.
\end{proposition}
\begin{remark}
The essence of this result is that even though the event $\{\Delta=n\}$ depends
on the whole random graph $G_p=G_p(-\infty)$, it has the same probability as  the
the event $H_{0,n}$ that depends only on the restriction of
the graph on the set $[0,n]$.
\end{remark}
With a quite similar argument, we also have that
\begin{proposition}
\label{phiH}
If $\phi(G_{p,\Gamma_0,\Gamma_1})$ is a deterministic real-valued functional
of $G_{p,\Gamma_0,\Gamma_1}$ then
\begin{align*}
\E \left[\phi(G_{p, \Gamma_0, \Gamma_1}) \vert \Gamma_0=0\right]
&= 
\sum_{n=1}^\infty 
\E \left[\phi(G_{p, 0, n}) ; H_{0,n} \right],
\end{align*}
provided that the expectation on the left exists.
\end{proposition}

The skeleton points $\mathcal S$ for $G_p=G_p(-\infty)$ remain skeleton points for
$G_p(x)$ for $-\infty < x < 2$ in the following sense:
\begin{lemma} \label{maxpathskel}
Let $-\infty < x <2$. 
If $\pi^* \in \Pi_{i,j}$ is maximal, 
that is, $w^x(\pi^*) = W^x_{i,j}$,
then $\pi^*$ contains all skeleton points between $i$ and $j$.
\end{lemma}
\begin{proof}
Let $s \in \mathcal S$, $i < s < j$, such that $s$ is not in $\pi^*$.
Then let $i_0$ be the largest vertex of $\pi^*$ below $s$ and $j_0$ the smallest
vertex of $\pi^*$ above $s$. Hence $(i_0, j_0)$ is an edge of $\pi^*$.
Since $s$ is a skeleton point there is a blue path $\pi'=(i_0, i_1, \ldots, i_k=s)$
from $i_0$ to $s$ and a blue path $\pi'' = (s, j_\ell, j_{\ell-1}, \ldots, j_0)$
from $s$ to $j_0$.
Consider now the path $\pi^{**}$ that contains the vertices of $\pi$
and of $\pi'$ and $\pi''$. 
We have
\[
w^x(\pi^{**}) = w^x(\pi^*) - w^x_{i_0,j_0} + k + \ell,
\]
since the edge $(i_0,j_0)$ of $\pi^*$ has been replaced by the edges
of $\pi'$ and $\pi''$ and the weights of $\pi'$ and $\pi''$ are $k$ and $\ell$
respectively because their edges have weight $1$ each.
Since $w^x_{i_0,j_0} < 2$
we have
$w^x(\pi^{**})  > w^x(\pi^*) - 2 + k +\ell  \ge w^x(\pi^*)$, contradicting the fact that $\pi^*$ is maximal.
Hence $s$ must belong to $\pi^*$.
\end{proof}

\subsection{Scaling property and side derivatives}
The following scaling property allows us to treat cases $x > 1$ as well. Recall that
$q=1-p$.
\begin{proposition}[Scaling property of $C_p(x)$]
\label{scaling1}
For $x>0$,
\[
C_p(x) = x C_q(1/x).
\]
\end{proposition}
\begin{proof}
For $i<j$ write \eqref{xweight} as
$w^x_{i,j} = x \left[\overline\alpha_{i,j} + \frac{1}{x} (1-\overline\alpha_{i,j})\right]$,
where 
\[
\overline\alpha_{i,j}:=1-\alpha_{i,j}.
\]
Hence the weight of edge $(i,j)$ in $G_p(x)$ is $x$ times its
weight in $G_q(1/x)$. 
Using \eqref{Wxij} and Theorem \ref{thmx} we conclude that
$C_p(x) = x C_q(1/x)$.
\end{proof}

We pass on to some preliminary analytical properties of the
function $x \mapsto C_p(x)$. We first obtain a different expression
for the function that is a consequence of Lemma \ref{maxpathskel}
and standard renewal theory. Indeed, due to  Lemma \ref{maxpathskel}
we can write the maximum weight $W^x_{0,n}$ of all paths in $\Pi_{0,n}$
as the sum of maximum weights of paths in $\Pi_{\Gamma_{k-1},\Gamma_k}$,
the sum taken over $k\ge 1$ such that $\Gamma_k \le n$,
plus the maximum weight of paths in $\Pi_{0,\Gamma_1}$,
plus the maximum weight of paths in  $\Pi_{\Gamma_k,n}$.
By the strong law of large numbers (see also \cite{FK03,FK18} for similar arguments)
we obtain
\begin{proposition}
\label{Csum}
For $x < 2$,
\[
C_p(x) = \gamma \E[ W^x_{\Gamma_1, \Gamma_2}]
= \gamma \E[W^x_{\Gamma_0, \Gamma_1}|\Gamma_1=0]
= \gamma \sum_{n=1}^\infty \E[W^x_{0, n}; H_{0,n}].
\]
\end{proposition}
The latter equality is due to Proposition \ref{phiH}.

\begin{corollary}
\[
\lim_{x \to -\infty} C_p(x) = C_p(-\infty).
\]
\end{corollary}

\begin{proposition}
The function $C_p(x)$ is convex over $x \in \R$.
\end{proposition}
\begin{proof}
By Theorem \ref{thmx}, $C_p(x) = \lim_{n \to \infty} \E W^x_{0,n}/n
= \sup_n \E W^x_{0,n}/n$. By \eqref{xweight}, \eqref{wpath} and \eqref{Wxij},
the function $x \mapsto W^x_{0,n}$ is a.s.\ the supremum of affine functions
and hence convex. Therefore $x \mapsto \E W^x_{0,n}/n$ is convex
and so $x \mapsto C_p(x)$ is convex, being the supremum of convex functions.
\end{proof}

\begin{corollary}
\[
\lim_{x \to \infty} \frac{C_p(x)}{x} = C_{1-p}(0).
\]
\end{corollary}
\begin{proof}
Convexity implies continuity.
The result then follows from the scaling property and continuity at $0$.
\end{proof}

Since $C_p(x)$ is convex, left and right derivatives exist.
Letting $D^-, D^+$ denote left and right differentiation, respectively, we
have 
\[
D^{\pm} C_p(x) = \gamma \E D^{\pm} W^x_{\Gamma_1, \Gamma_2},
\]
due to the dominated convergence theorem that is easily justifiable.
See \eqref{justdct} below.


Our goal is to identify all  points at which the left and right derivatives 
of $C_p(\cdot)$ differ. For $\pi \in \Pi_{\Gamma_1, \Gamma_2}$ let
us write its weight as 
\[
w^x(\pi) 
= N_{G_p}(\pi) + x \overline N_{G_p}(\pi),
\]
where $N_{G_p}(\pi)$ is the number of blue edges of $\pi$ (that is,
the number of edges of $\pi$ that are also edges in $G_p$),
and $\overline N_{G_p}(\pi)$ the number of red edges (the number of edges of $\pi$ that are not
edges in $G_p$). 
Then consider 
\[
\Pi^{*,x}_{\Gamma_1,\Gamma_2} =\{\pi \in \Pi_{\Gamma_1,\Gamma_2}:\,
w^x(\pi) = W^x_{\Gamma_1,\Gamma_2}\},
\]
the set of paths with maximal weight. Then
\[
D^+ W^x_{\Gamma_1, \Gamma_2} 
= \max_{\pi  \in\Pi^{*,x}_{\Gamma_1,\Gamma_2}} \overline N_{G_p}(\pi)
\qquad 
D^- W^x_{\Gamma_1, \Gamma_2} 
= \min_{\pi  \in\Pi^{*,x}_{\Gamma_1,\Gamma_2}} \overline N_{G_p}(\pi).
\]
This is rather trivial: all we are saying is that if the function $\phi$ is the maximum
of affine functions, say, 
$\phi(x)=\max_j (a_j+b_j x)$, then its right (respectively, left) derivative at $x$
equals the maximum (respectively, minimum) of all $b_j$ 
such that  $a_j+b_jx = \phi(x)$.
The only thing we did is to translate this
obvious fact in our notation. Since
\begin{equation}
\label{justdct}
|D^\pm W^x_{\Gamma_1,\Gamma_2}| \le \Gamma_2-\Gamma_1,
\end{equation}
and since $\E (\Gamma_2-\Gamma_1)=1/\gamma<\infty$,
the dominated convergence theorem applies and so
\[
D^+C_p(x)= 
  \gamma \E D^+ W^x_{\Gamma_1,\Gamma_2}
=\gamma \E   \max_{\pi  \in\Pi^{*,x}_{\Gamma_1,\Gamma_2}} \overline N_{G_p}(\pi)
\]
Similarly,
\[
D^-C_p(x) =  \gamma \E D^- W^x_{\Gamma_1,\Gamma_2}
=\gamma \E   \min_{\pi  \in\Pi^{*,x}_{\Gamma_1,\Gamma_2}} \overline N_{G_p}(\pi).
\]



As a consequence of the above we obtain the auxiliary result:
\begin{lemma}
\label{irrdiff}
If $x$ is irrational then $C_p$ is differentiable at $x$.
\end{lemma}
\begin{proof}
By the scaling property, it suffices to show the claim for $x<1$.
Consider the expression for $C_p(x)$ from Proposition \ref{Csum}.
The set of points at which $x \mapsto W^x_{\Gamma_1,\Gamma_2}$ fails to be differentiable
is included in the set of points $x$ for which there are two paths $\pi_1, \pi_2$
from $\Gamma_1$ to $\Gamma_2$
such that $w^x(\pi_1)=w^x(\pi_2)$ with $\overline N_{G_p}(\pi_1) \neq \overline N_{G_p}(\pi_2)$.
This implies that $(\overline N_{G_p}(\pi_2)-\overline N_{G_p}(\pi_1)) x =  N_{G_p}(\pi_1)- N_{G_p}(\pi_2)$,
that is, that $x$ is rational. 
Hence  the left and right derivatives of 
$x \mapsto W^x_{\Gamma_1,\Gamma_2}$ coincide at irrational points.
Hence the left and right derivatives of 
$x \mapsto C_p(x)$ coincide at irrational points.
\end{proof}

To precisely identify the points of nondifferentiability we define
the notion of criticality.

\section{Criticality and nondifferentiability}

By directed graph $G$ on $[0,n]=\{0,1,\ldots,n\}$ we mean any graph with edge
directions compatible with the natural integer ordering.
Let $\fG_n$ be the set of all directed graphs on $[0,n]$.
\begin{definition}
\label{hndef}
Let $\fH_n$ be the set of  all directed graphs $G \in \fG_n$ such that
for all $j \in [1,n-1]$,
\\[1mm]
1\o) there is a path in $G$ from $0$ to $n$ containing $j$;
\\
2\o) for some  $i\neq j$ there is no path in $G$ from $\min(i,j)$ to $\max(i,j)$.
\end{definition}

\begin{remark}
The set $\fH_n$ is nonempty for all  positive integers 
$n \neq 2$ but $\fH_2=\varnothing$.
For $n \ge 3$, every $G \in \fH_n$ contains the edges $(0,1)$ and $(n-1,n)$.
\end{remark}

\begin{remark}
For the Barak-Erd\H{o}s random directed graph $G_p$, let 
$\{G_{p,0,n} \in \fH_n\} \subset \Omega$ be the event such that
$G_p$, restricted on $[0,n]$, is in the class $\fH_n$. Then
\[
\{G_{p,0,n} \in \fH_n\} = H_{0,n},
\]
where $H_{0,n}$
is the event appearing in Proposition \ref{deltaH}.
\end{remark}

For $\pi \in \Pi_{0,n}$ and $G \in \fG_n$ we let  $N_G(\pi)$ be the number of
edges of $\pi$ that are also edges of $G$, and $\overline N_G(\pi)$ be the number of
edges of $\pi$ that are not edges of $G$.

\begin{definition}
\label{defcrit}
We say that $x \in \R$ is {\bf\em critical} if there is a positive integer $n$
and and a graph $G \in \fH_n$ possessing two paths $\pi_1, \pi_2$
such that
\\[1mm]
1\o) $N_G(\pi_1) + x \overline N_G(\pi_1) =
N_G(\pi_2) + x \overline N_G(\pi_2) 
= \max_{\pi \in \Pi_{0,n}} (N_G(\pi) + x \overline N_G(\pi))$
\\
2\o) 
$\overline N_G(\pi_1)  \neq \overline N_G(\pi_2)$.
\end{definition}

\begin{remark}
\label{x1x}
Note that, for $x>0$, if $x$ is critical then $1/x$ is also critical
because, in the definition of criticality, 
we can replace $G$ by the graph $\overline G$ whose edges
are the nonedges of $G$.\footnote{A {\emph nonedge} $(i,j)$ 
of $G$, where $i,j$ are vertices of $G$, means that
$(i,j)$ is not an edge of $G$.}
\end{remark}
\begin{remark}
If $x\neq 0$ is critical then 2\o) of Def.\ \ref{defcrit}
can be replaced by $N_G(\pi_1) \neq N_G(\pi_2)$.
\end{remark}
\begin{remark}
If $x$ is critical then the $n$ of  Def.\ \ref{defcrit} can be taken to be at least 3.
\end{remark}

\begin{theorem}
\label{thmdiffcrit}
$C_p(\cdot)$ is not differentiable at $x$ if and only of $x$ is critical.
\end{theorem}
\begin{proof}
It suffices to prove the statement for $x< 2$.
By Proposition \ref{Csum} and the dominated convergence theorem,
\[
D^{+} C_p(x) -D^{-} C_p(x)  
= \gamma \sum_{n=1}^\infty
\E \left [D^+ W^x_{0,n}-D^- W^x_{0,n}; H_{0,n} \right].
\]
Suppose that $x$ is critical. Let $n \ge 3$ and $G \in \fH_n$ be as in the definition
of criticality.
Since $H_{0,n} = \{G_p \in \fH_n\} \supset \{G_p = G\}$, we have
\[
D^{+} C_p(x) -D^{-} C_p(x)   \ge \gamma 
\E\left [D^+ W^x_{0,n}-D^- W^x_{0,n}; G_p=G\right].
\]
The event $\{G_p=G\}$ is simply the event that for all edges $(i,j)$ of $G$
we have $\alpha_{i,j}=1$, whereas for all nonedges $(i,j)$ we have $\alpha_{i,j}=0$.
Obviously, on this event,
$w^x(\pi) = N_{G_p}(\pi) +x \overline N_{G_p}(\pi)= N_G(\pi) +x \overline N_G(\pi)$ for all $\pi \in \Pi_{0,n}$.
Since 
$W^x_{0,n} = \max_{\pi \in \Pi_{0,n}} (N_{G_p}(\pi)+x \overline N_{G_p}(\pi))$,
we have
\begin{equation}
\label{dd}
D^+ W^x_{0,n} - D^- W^x_{0,n} 
= \max \overline N_G(\pi) - \min \overline N_G(\pi), \quad\text{ on } \{G_p=G\},
\end{equation}
where both the $\max$ and the $\min$ are taken
over all $\pi \in \Pi_{0,n}$ such that 
$N_G(\pi) +x \overline N_G(\pi) = W^x_{0,n}$.
Let $\pi_1, \pi_2$ be as in the definition of criticality.
Then $N_G(\pi_1) + x \overline N_G(\pi_1) =
N_G(\pi_2) + x \overline N_G(\pi_2)=W^x_{0,n}$ 
and $\overline N_G(\pi_1) 
\neq \overline N_G(\pi_2)$.
Hence
\[
D^+ W^x_{0,n} - D^- W^x_{0,n}  
\ge \left|\overline N_G(\pi_1) -\overline N_G(\pi_2)\right|>0,
\quad \text{on }  \{G_p=G\}.
\]
Since $\P(G_p=G)>0$, we conclude that $D^{+} C_p(x) -D^{-} C_p(x)>0$
if $x$ is critical.
Conversely, if $C_p$ is not differentiable at $x$ then there is $n$
such that $\E \left [D^+ W^x_{0,n}-D^- W^x_{0,n}; H_{0,n} \right]>0$.
Hence $\P(D^+ W^x_{0,n}-D^- W^x_{0,n}>0 ;H_{0,n})>0$.
Then there exists $\omega_0 \in H_{0,n}$
such that $D^+ W^x_{0,n}(\omega_0)-D^- W^x_{0,n}(\omega_0)>0$.
But $\omega_0 \in H_{0,n}$ is equivalent to $G_p\equiv G_p(\omega_0) \in \fH_n$.
For this $\omega_0$, let $\widehat G(\omega_0)$ be the graph with edges precisely those
$(i,j)$ for which $\alpha_{ij}(\omega_0)=1$. Then $\widehat G(\omega_0) \in \fH_n$.
Using \eqref{dd} we obtain that the conditions of Definition \ref{defcrit}
are satisfied with $G=\widehat G(\omega_0)$.
\end{proof}




\section{Identifying critical and noncritical points}
We have reduced the problem of finding the points of nondifferentiability of $C_p(\cdot)$=-=
to the problem of {\black finding all critical points in the sense of Definition \ref{defcrit}.}
This is a graph-theoretic, completely deterministic, issue that we tackle in this section.
For $x \in \R$, $G \in \fG_n$ and $\pi \in \Pi_{0,n}$ {\black we use the
term $(x,G)$-weight for the quantity
$w^x_G(\pi) = N_G(\pi) + x \overline N_G(\pi)$;}
we say that {\black $\pi$ is $(x,G)$-maximal (or, simply, maximal) if $w^x_G(\pi)
\ge w^x_G(\pi')$ for all $\pi' \in \Pi_{0,n}$.}
Whenever no confusion arises, we omit the superscript $x$ and the subscript $G$
from the symbols above.

{\black To show that an $x$ is critical} we will proceed by giving an explicit construction of an 
appropriate graph.

%

{\black To show that $x$ is not critical }we must show, for every 
$n \ge 3$ and  every $G \in \fH_n$, that
either there is a unique $(x,G)$-maximal path
or every $(x,G)$-maximal path has the same $\overline N_G(\pi)$.

\begin{theorem}
\label{thmcrit}
The set of critical points is the union of
\\
1) nonpositive rationals;
\\
2) positive integers except $1$;
\\
3) the reciprocals of positive integers except $1$.
\end{theorem}


The theorem follows from a number of intermediate results.
We point out that Lemmas \ref{1noncrit} and \ref{0crit} are special cases of 
Propositions \ref{+noncrit01} and \ref{-ncrit}, respectively. 

\begin{lemma}
\label{1noncrit}
$1$ is not critical.
\end{lemma}
\begin{proof}
For every $n \ge 3$, every path $\pi\in \Pi_{0,n}$, and every $G \in \fG_n$, we have
$w^1_G(\pi)=|\pi|$, the length of $\pi$.
The maximum of $|\pi|$ over all paths from $0$ to $n$ is obviously $n$.
Clearly, the only path with length $n$ is the path $(0,1,2,\ldots,n)$.
\end{proof}

\begin{lemma}
\label{0crit}
$0$ is critical.
\end{lemma}
\begin{proof}
For any $n \ge 3$, any path $\pi \in \Pi_{0,n}$ and any graph $G \in \fG_n$,
we have
$w^0_G(\pi) =N_G(\pi)$.
Let $n=3$ and let  the edge set of the graph $G$  be\\
\begin{minipage}{0.6\textwidth}
\[
\edges(G) = \big\{(0,1), (1,3), (0,2), (2,3) \big\}
\]
\end{minipage}
\begin{minipage}{0.4\textwidth}
\includegraphics[width=4cm]{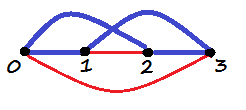}
\end{minipage}
\\
Clearly, $G \in \fH_3$.
Consider the paths in $\Pi_{0,3}$. There are just 4 paths:
 the path $(0,3)$ of length 1, the paths $(0,1,3)$ and $(0,2,3)$ of length 2,
and the path $(0,1,2,3)$ of length 3.
Considering all 4 possible paths in $G$, we easily
see that $W^0_G=2$ and this is achieved by 
$\pi_1=(0,1,3)$ and $\pi_2=(0,1,2,3)$.
Since $\overline N_G(\pi_1)=0 \neq \overline N_G(\pi_2)=1$, we conclude that
$x=0$ is critical.
\end{proof}


\begin{proposition}
\label{+crit}
For any positive integer $k \ge 2$, $k$ and $1/k$ are critical.
\end{proposition}
\begin{proof}
By Remark \ref{x1x}, it is enough to show the criticality of $x=1/k$ for
some integer $k \ge 2$.
We will take $n =k+2$ and exhibit a graph $G \in \fH_n$ satisfying the condition
of the definition of criticality.
Consider the graph $G$ with edges
\\
\begin{minipage}{0.5\textwidth}
\[
\edges(G) = \bigcup_{i=1}^{n-1} \big\{(0,i), (i,n)\big\} \cup \big\{(1,n-1)\big\}
\]
\end{minipage}
\begin{minipage}{0.5\textwidth}
\includegraphics[width=7cm, height=3cm]{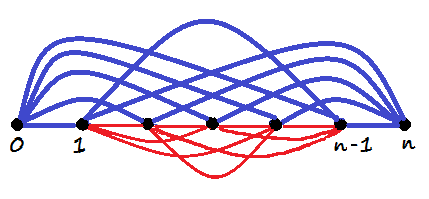}
\end{minipage}
\\
It is easy to see  that $G \in \fH_n$.
Indeed, for every $i \in [1,n-1]$ the sequence $(0,i,n)$ is a path in $G$
containing $i$;
if $1 \le j  \le n-2$ then there is no path in $G$ from $j$ to $j+1$;
for $j=n-1$ there is no path in $G$ from $n-2$ to $n-1$.
We now show that $W^{x}_G = 3$.
If $\pi \in \Pi_{0,n}$ has length at most $3$ then
$w^x_G(\pi) \le 3$ since the weight of each edge is at most 1.
If $\pi \in \Pi_{0,n}$ has length $\ell>3$ then
$\pi=(0,i_1, \ldots,i_{\ell-1}, n)$, and we see that 
$N_G(\pi)=2$, $\overline N_G(\pi)=\ell-2$,
so $w^x_G(\pi) = 2+ x(\ell-2) \le 2+ \frac{1}{k} (n-2) = 3$.
The path $\pi_1=(0,1,n-1,n)$ has $w^x_G(\pi_1)=3$.
So $W^x_G=3$, as claimed.
Consider also the path $\pi_2=(0,1,2,\ldots,n-1,n)$.
Again, $w^x(\pi_2)=3$ as well.
However, $N_G(\pi_1)=3$ but $N_G(\pi_2)=2$.
Hence $x=1/k$ is critical.
\end{proof}

\subsection{\black Properties of maximal paths; identifying noncritical positive points}
Showing noncriticality requires a bit more work. This relies on
identifying some properties of maximal paths.
We explain these properties in the four lemmas below
and then show that all positive real numbers, except those that are 
equal to $k$ or $1/k$ where $k \ge 2$ is an integer, are noncritical.
We need some auxiliary terminology:
\\[1mm]
{\black
$\bullet$
Every edge of the form $(i,i+1)$ is called {\bf\em short}. Otherwise, it is called {\bf\em long}.
\\
$\bullet$ We say that edge $e=(i,j)$ is {\bf\em  nested} in $e'=(i',j')$  
if $i' \le i < j \le j'$ and $e\ne e'$. 
}
\\[1mm]
The usefulness of this notion is as follows. Let $\pi' = (i', i, j, j')$ (maybe $i=i'$ or $j=j'$ here). Clearly, if $w^x(\pi')>w^x(e')$ then no maximal path may contain the edge $e'$. The next lemma specifies several cases when this condition holds.
To ease language, we think of all edges as being either blue (these
are the edges of $G$) or red (the nonedges of $G$).
Blue edges have weight $1$. Red edges have weight $x$.
So, for an arbitrary path $\pi$, $N_G(\pi)$, respectively $\overline N_G(\pi)$, 
is the number of blue, respectively  red, edges of $\pi$.

\begin{lemma}
\label{nested}
Assume the edge $e=(i,j)$ is nested in $e'=(i',j')$  and one of the following conditions holds:
\\
1) $x>0$ and $e$ and $e'$ have the same color;
\\
2)  $0<x<2$ and $e$ is blue;
\\
3) $0<x<2$ and $e'$ is red.
\\
Then $w^x(\pi')>w^x(e')$, where $\pi' = (i', i, j, j')$,
where we allow the possibility that $i'=i$ or $j'=j$.
\end{lemma}
\begin{proof}
Since $e\ne e'$, the path $\pi'$ contains at least one edge other than $e$. Denote it by $e''$. Then 
\[
w^x(\pi')-w^x(e') \ge w^x(e)+w^x(e'')-w^x(e')
\ge w^x(e)+\min\{x,1\} -w^x(e').
\]
1) If $e$ and $e'$ are of the same color then $w^x(e)=w^x(e')$, so the right-hand side
of the above display equals  $\min\{x,1\}$ which is positive.
\\
2) If $e$ is blue then $w^x(e)=1$ and so
$w^x(\pi')-w^x(e')\ge 1 -\max\{x,1\}+\min\{x,1\} = \min\{x,2-x\}>0$, 
since $0<x<2$.
\\
3) If $e'$ is red then $w^x(e')=x$ 
and so $w^x(\pi')-w^x(e') \ge 2\min\{x,1\} - x = \min\{x,2-x\}>0$, 
since $0<x<2$.
\end{proof}

\begin{lemma}
\label{shortfact}
If $0<x<2$ then every maximal path contains all short blue edges.
\end{lemma}
\begin{proof}
Let $\pi$ be a maximal path and assume there is a short blue edge
$e=(j,j+1)$ which is not an edge of $\pi$. Then the path $\pi$ must contain an edge 
$e'=(i',j')$ such 
that $e$ is nested in $e'$. Let $\pi'=(i', j, j+1, j')$.
Then by the second part of Lemma~\ref{nested}, 
$w^x(e')<w^x(\pi')$. 
This contradicts the maximality of $\pi$
because we can replace the edge $e'$ of $\pi$ by $(i',i,j,j')$ and
obtain a path with weigh strictly larger than the weight  $\pi$.
\end{proof}

\begin{lemma}
\label{longfact}
If $0<x<2$ 
then every long edge of a maximal path $\pi$ is blue
(
in other words, every red edge of $\pi$ must be short).
\end{lemma}
\begin{proof}
Let $\pi$ be a maximal path having a long edge $e'=(i',j')$. Since $j'-i' \ge 2$,
there exists an edge $e=(i,j)$ that is nested in $e'$. If $e'$ is red then  by the third condition of Lemma~\ref{nested}, 
$w^x(e')<w^x(i',i,j,j')$, contradicting the maximality of $\pi$. 
\end{proof}

\begin{remark}
Lemmas \ref{shortfact} and \ref{longfact} tell us that  if $0<x<2$ then
the edges of a maximal path are classified as follows;
Long edges: they are all blue.
Short edges: they include every possible blue edge and, possibly, some red ones.
\end{remark}

The next lemma follows directly form the first condition of Lemma~\ref{nested}. 

\begin{lemma}
\label{nestfact}
If $x>0$ then no blue edge of a maximal path can be nested in a 
different blue edge of another
maximal path.
\end{lemma}

{\black Let $\pi, \pi' \in \Pi_{0,n}$.
We say that the interval $[i,j] \subset [0,n]$ is {\emph $(\pi,\pi')$-special}
if the set of vertices $k \in [i,j]$ that belong to both $\pi$ and $\pi'$
consists of $i$ and $j$ only.}

\begin{lemma}
\label{specialfact}
Let $0<x<2$. Then for every pair $\pi, \pi' \in \Pi_{0,n}$ of maximal paths such that 
$\overline N_G(\pi)\neq \overline N_G(\pi')$ 
there is a $(\pi,\pi')$-special interval $I$ such that 
\[
\overline N_G(\pi|_{I})
\neq \overline N_G(\pi'|_{I})
\] 
and 
\[
N_G(\pi|_{I})-N_G(\pi'|_{I}) \in\{-1,1\}.
\]
\end{lemma}
\begin{proof}
Let $\pi, \pi'$ be two $(x,G)$-maximal paths, where $0<x<2$ and $G \in \fG_n$.
Since $x$ and $G$ are fixed throughout this proof, we omit them in superscripts
or subscripts and, for example, simply write $w$ instead of
$w^x_G$.
Let
\[
0=s_0 < s_1 < \cdots < s_r=n
\]
be the common vertices of the two paths and let $\pi_k, \pi_k'$ be the restrictions of
$\pi, \pi'$, respectively, on the set of vertices $[s_{k-1}, s_k]$,
$k=1,\ldots, n$. Any of these intervals is $(\pi,\pi')$-special, but
we  pick it in a way that 
\begin{equation}
\label{NN}
\overline N(\pi_k) \neq \overline N(\pi'_k).
\end{equation} 
Indeed, if we had $\overline N(\pi_k) = \overline N(\pi'_k)$ for all $k$,
we would have $\overline N(\pi) = \sum_{k=1}^r \overline N(\pi_k)
= \sum_{k=1}^r \overline N(\pi_k') = \overline N(\pi')$, 
in contradiction to the assumption.
On the other hand we have
\begin{equation}
\label{ww}
w(\pi_k) = w(\pi'_k).
\end{equation}
Indeed, if there were a $k$ for which $w(\pi_k) < w(\pi'_k)$ then
we would have obtained a path of weight strictly larger than $w(\pi)$
by simply replacing the subpath $\pi_k$ with $\pi'_k$ in $\pi$.
We now focus on the $(\pi,\pi')$-special interval
\[
I=[s_{k-1}, s_k] =: [i,j]
\]
and show that $N(\pi_k) - N(\pi_k') = \pm 1$.
We already know that 
\begin{equation}
\label{not0}
N(\pi_k) - N(\pi_k')  \neq 0
\end{equation}
because 
of \eqref{NN}, \eqref{ww} and the assumption that $x>0$.
Consider the structure of maximal paths, as shown in Lemmas \ref{shortfact},
\ref{longfact} and \ref{nestfact}. 
For either of the paths $\pi$, $\pi'$, every long edge is blue and every short edge is red.
A typical maximal path is has blue (long) edges interlaced by intervals of
red (short) edges.
Bear in mind that the latter intervals may be just a single point.
See Figure \ref{twopaths} for an illustration.
Let $\NN$ be the set of blue edges of $\pi_k$. Similarly, $\NN'$ for $\pi_k'$.
To prove the claim, we need to define the following set of intervals:
\begin{align*}
&\GG := \GG_\internal \cup \GG_\first \cup \GG_\last, \text{ where }
\\
&\GG_\internal
:=\big\{ [u,v]:\, i<u \le v<j ~\exists a,b \text{ such that } (a,u), (v,b) \in \NN \big\},
\\
&\GG_\first:=\big\{[i,v]:\, i<v \text{ and } \exists b \text{ such that }  (v,b) \in \NN \big\},
\\
&\GG_\last:=\ \big\{[u,j]:\, u<j \text{ and } \exists a \text{ such that }  (a,u) \in \NN \big\}.
\end{align*}
{\black Note that all intervals $[u,v], [i,v],$ and $[u,j]$ mentioned above consist 
of short red edges in the path $\pi$.}
We similarly define $\GG'$ for $\pi_k'$.
Elements of $\GG$ are referred to as ``gaps'', meaning spaces between successive
blue edges.
Note that a element $[u,v]$ of  $\GG_\internal$ may be a single point if $u=v$
and is the point where a blue edge finishes and another one starts.
If it is not a point then, necessarily, the short edges on that interval are red.
The elements of $\GG_\internal$ are ``internal intervals''.
Note also that the set $\GG_\first$ is either empty (if the first blue edge starts at $i$)
or a singleton.
In the latter case, the short edges within the unique interval $[i,v]$ of $\GG_\first$
are red.
Similarly $\GG_\last$ is either empty or a singleton containing an interval
of the form $[u,j]$ such that all the short edges in it are red.
Let
\[
\Gamma : = |\GG|
\]
and notice that
\begin{equation}
\label{R-N}
\Gamma-N \in \{0,-1,1\}.
\end{equation}
We establish a bijection $\phi : \GG \to \NN'$, from the set of gaps of $\pi_k$
onto the set of blue intervals of $\pi_k'$ that can simply be described as follows.
If $g \in \GG$, there is a unique {\black $e' \in \NN'$} such that $g$ is strictly nested in $e'$.
We need to show that this is possible. 
Let $g=[u,v] \in \GG_\internal$. 
If $u=v$ then, by virtue of the fact that $\pi_k$ and $\pi_k'$ have no common
vertices other than $i$ and $j$, there is a blue edge $e'=(a',b') \in \NN'$
such that $a' < u=v < b'$.
If $u< v$ then there can be no gap $[u',v'] \in \GG'$ such that $[u,v] \cap [u',v']
\neq \varnothing$ otherwise the two paths would have common internal vertices.
Again, there is a unique $e'=(a',b') \in \NN'$ such that $a' < u < v < b'$.
If $g = [i,v] \in \GG_\first$ then $i<v$.
Let $(a',b') \in \NN'$ be the leftmost  blue edge of $\pi_k'$.
We cannot have $a'>i$ because, in this case, $[i,a'] \in \GG'_\first$
which would imply that $[i,v] \cap [i,a'] \neq \varnothing$ and then
the two paths would have common internal vertices, which is impossible.
Hence $a'=i$.
The endpoint $b'$ of the blue edge $[i,b]$ must be strictly larger than $v$,
otherwise, again, the two paths would have had common internal vertices.
Hence if $g \in \GG_\first$ there is $e' \in \NN'$ such that $g$ is strictly nested in $e'$.
Completely symmetrically, we have that
if $g \in \GG_\last$ there is $e' \in \NN'$ such that $g$ is strictly nested in $e'$.
The mapping
\[
\phi : \GG \to \NN'
\]
has thus been constructed.
{\black This mapping is one-to-one and onto by Lemma 9 and since the interval $[i,j]$ is $(\pi,\pi')$-special.}
We therefore have
\begin{equation}
\label{equal}
\Gamma = N'.
\end{equation}
We now consider the three possible values of $\Gamma-N$ as in \eqref{R-N}.
If $\Gamma-N=0$ then \eqref{equal} gives $N=N'$, in contradiction to \eqref{not0}.
If $\Gamma-N=\pm 1$ then \eqref{equal} gives $N'-N=\pm 1$.
This concludes the proof.
\begin{figure}[h] 
\includegraphics[width=11cm]{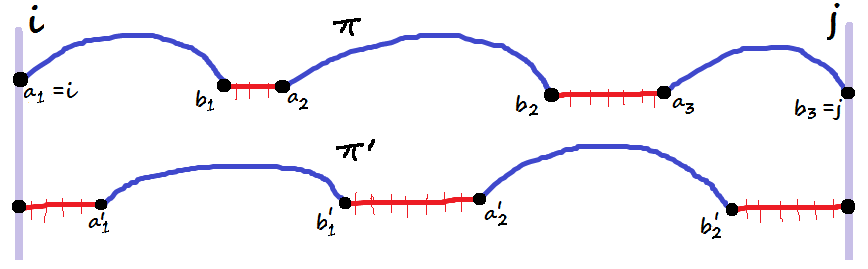}
\captionof{figure}{Two maximal paths on an interval $[i,j]$
having no common vertices other than $i$ and $j$}
\label{twopaths}
\end{figure}
\end{proof}

\begin{proposition}
\label{+noncrit01}
Every $0<x<2$ that is not  the reciprocal of an integer is not critical.
\end{proposition}
\begin{proof}
We prove the contrapositive: if $0<x<2$ is critical then $x=1/m$ for
some integer $m$.
So suppose that $x$ is critical and $0<x<2$.
Then there is $n\ge 3$ and $G \in \fH_{n}$ 
(edges of $G$ are called blue and nonedges red) and two maximal paths $\pi_1$,
$\pi_2$ with different number of red edges: $\overline N_G(\pi_1)\neq \overline N_G(\pi_2)$.
By the Lemma \ref{specialfact}, there is a $(\pi_1, \pi_2)$-special interval $[i,j]$
such that  $N_G(\pi_1|_I)-N_G(\pi_2|_I)\in \{1,-1\}$.
Since 
\[
0=w^x_G(\pi_1|_I)-w^x_G(\pi_2|_I) = (\overline N_G(\pi_1|_I)- \overline N_G(\pi_2|_I))x
+ (N_G(\pi_1|_I)-N_G(\pi_2|_I)),
\] 
it follows that
\[
x=\frac{1}{|\overline N_G(\pi_1|{[i,j]})- \overline N_G(\pi_2|{[i,j]})|},
\]
and hence the reciprocal of a positive integer.
\end{proof}

\subsection{\black Criticality of negative rationals}
We finally show that negative rational numbers are critical.
This is done via an explicit construction of an appropriate graph.
We deal with negative integers first.
\begin{proposition}
\label{-ncrit}
Any negative integer is critical.
\end{proposition}
\begin{proof}
Let $-\ell$ be a negative integer.
Let $n=2\ell+3$.
Define $G$ by listing its edges:
\[
\edges(G) = \bigcup_{i=0}^\ell\big\{(i,i+1)\big\} \cup
\bigcup_{i=\ell+2}^{2\ell+2} \big\{(i,i+1)\big\} \cup
\big\{(0,\ell+2), (\ell+1, 2\ell+3)\big\}.
\]
\begin{center}
\includegraphics[width=9cm]{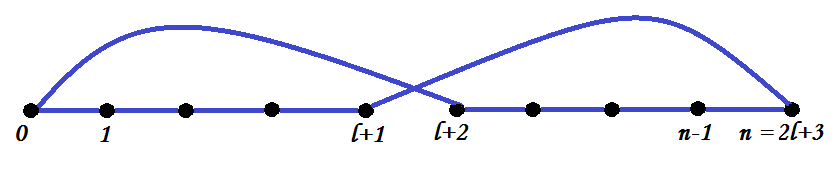}
\end{center}
\noindent
Note that $G \in \fH_n$ because, for all $j \in [1,n-1]$ there is a path in $G$ from that contains
$0$, $j$ and $n$ as vertices.
Moreover, every $j \le \ell+1$ is not connected to $\ell+2$
and every $j \ge \ell+2$ is not reachable from $\ell+1$.
We see that the maximal weight over all paths is
\[
W^{-\ell}_G=\ell+2
\]
and is achieved by the path
$\pi=(0,1,\ldots,\ell+1,n)$ that contains only blue edges.
On the other hand, the path $\pi'$ that contains all vertices between $0$ and $n$,
namely $\pi'=(0,1,\ldots,\ell+1,\ell+2, \ldots,n)$, has weight $w^{-\ell}_G(\pi')
= (\ell+1)\cdot 1 + (-\ell) + (\ell+1) \cdot 1 =\ell+2$,  that is, it has 
maximal weight. Notice that $\pi'$ contains exactly one red edge: the edge $(\ell+1, \ell+2)$.
The paths $\pi$ and $\pi'$ both achieve the maximal weight $\ell+2$ but contain
a different number of red edges ($\pi$ has no red edges but $\pi'$ has $1$). 
Hence $-\ell$ is critical.
\end{proof}

To show the criticality of negative rational numbers that are not integers we use
a consequence of Sturm's lemma \cite{LOT}  
to construct the graph $G$ in the definition of criticality.
Since this consequence is not proved in \cite{LOT} we provide a proof below.

\begin{lemma}[Corollary to Sturm's lemma]
\label{sturmlem}
Let $N, n$ be positive integers,  $N\ge n$. Then
there exists a unique finite sequence $v=(v_1,\ldots,v_N)$ of elements 
of $\{0,1\}$ such that, for all $0\le i<j\le N$, 
the following hold:
\begin{align}
& \sum_{i < k \le j} v_k \in \bigg\{ \bigg\lfloor (j-i) \frac{n}{N}  \bigg\rfloor, \,  \bigg\lceil (j-i) \frac{n}{N} \bigg\rceil \bigg\} , 
\label{UP}
\\
& \sum_{i < k \le N} v_k = \bigg\lfloor (N-i) \frac{n}{N}  \bigg\rfloor,
\label{DOWN}
\\
& v_1=1. \label{ONE}
\end{align}
\end{lemma}

\begin{proof}
Fix integers $1 \le n \le N$. {\black 
Put
\[
\label{vk}
v_k=\bigg\lfloor (N-k+1) \frac{n}{N}  \bigg\rfloor - \bigg\lfloor (N-k) \frac{n}{N}  \bigg\rfloor.
\]
Since $(N-k+1)n/N=(N-k)n/N+n/N$ and $n\le N$, we have $v_k\in \{0,1\}$ for all $k$. Clearly, $v_1=1$. Note that
\[
\sum_{i < k \le j} v_k = \bigg\lfloor (N-i) \frac{n}{N}  \bigg\rfloor - \bigg\lfloor (N-j) \frac{n}{N} \bigg\rfloor 
\]
\[
=(N-i) \frac{n}{N}-\alpha - ((N-j) \frac{n}{N} - \beta) = (j-i) \frac{n}{N} + (\beta-\alpha),
\]
for some $\alpha,\beta \in [0,1)$. But then $\beta-\alpha \in (-1,1),$ that is, \eqref{UP} holds.
In particular, 
\[
\sum_{i < k \le N} v_k = \bigg\lfloor (N-i) \frac{n}{N}  \bigg\rfloor,
\]
finishing the proof.
}
\end{proof}

We use the terminology {\emph ``$(N,n)$-balanced sequence''}
for a sequence $v=(v_1,\ldots,v_N)$ satisfying the conditions 
\eqref{UP}, \eqref{DOWN} and \eqref{ONE} of Lemma \ref{sturmlem}.

\begin{proposition}
\label{-qcrit}
Every negative rational number that is not an integer is critical.
\end{proposition}
\begin{proof}
{\black Let $x$ be} a strictly negative rational but not an integer. We can then write
\[
x=-\ell+ \frac{s}{t},
\] 
where $\ell, s, t$ are positive integers, $t>1$, 
with $s,t$ coprime, $s<t$, $\ell t -s > 0$.
Define
\begin{align*}
m &: = t(\ell+3)-(s+1)
\\
n &:=3m.
\end{align*}
For this $n=n(x)$, we shall exhibit a graph $G = G(x) \in \fG_n$ by first defining a set 
of special vertices $a_0, \ldots, a_t$ and then by defining its edges.
Let $(v_1, v_2,\ldots, v_t)$ be the $(t,t-s)$-balanced sequence, 
as in Lemma \ref{sturmlem}. 
Thus, $v_1=1$, $v_t=0$ and $\sum_{j=1}^t v_j = t-s$.
Then set
\begin{align*}
a_0 &:=m, \\
a_1 &:= a_0+ (\ell+1) + v_1, \\
a_j &:= a_{j-1} + (\ell+2) + v_j, \quad j=2,\ldots, t.
\end{align*}
Note that
\[
a_0 = a_t-a_0 = n-a_t =m.
\]
Indeed,
\begin{multline*}
a_t-a_0 = \sum_{j=1}^t (a_j-a_{j-1})
= (\ell+1+ v_1) + \sum_{j=2}^t (\ell+2+v_j)
\\
=(\ell+1) + (t-1) (\ell+2) + (t-s) = t(\ell+3)-(s+1)=m.
\end{multline*}
Hence the vertex set $[0,n]$
is split into three sections, $[0,a_0]$, $[a_0, a_t]=[m,2m]$, $[a_t, n]$, 
of length $m$ each.
We use the term {\emph pivot vertices} for the vertices $a_0, \ldots, a_t$.
We now define the edge set $\edges(G)$ of $G$ as the union of 
\begin{align*}
&\edges_1 = 
\big\{(i,i+1):\, i \in [0, n-1] \setminus\{a_0,a_1-1,a_1,a_1+1,\ldots,a_{t-1}-1,a_{t-1},a_{t-1}+1, a_t-1\}\big\},
\\
&\edges_\piv = \big\{(a_0,a_{1}),\, (a_1,a_2),\ldots, (a_{t-1},a_t)\big\} ,
\\
&\edges_\hop=\big\{(a_j-1,a_j+1):\,{\black 1} \le j \le t-1\big\}\cup \{(a_t-1, a_t)\},
\\
&\edges_\hyper =  \big\{(0,a_0+1)\big\}
\cup \big\{ (0, a_j+2):\, 1 \le j \le t-1\big\}
\cup \big\{ (a_j+1,n):\, 1 \le j \le t-1 \big\}.
\end{align*}
\begin{figure}[h] 
\includegraphics[width=\textwidth,height=3cm]{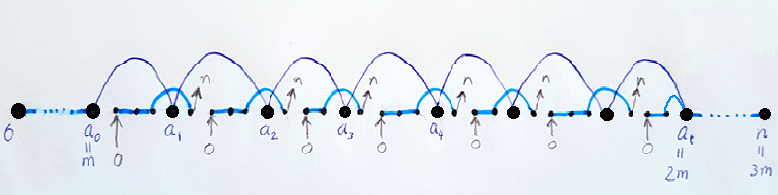}
\captionof{figure}
{The graph $G$ for the proof of Proposition \ref{-qcrit}. 
The pivot vertices are $a_0, \ldots, a_t$.
Edges between successive pivots comprise the set $\edges_\piv$. 
Edges of the form $(a_j-1, a_j+1)$, $j=1,\ldots,t-1$ and the edge $(a_t-1,a_t)$
comprise the set $\edges_\hop$. 
The edges in $\edges_\hyper$ are not all shown. 
The segments $[0,m]$ and $[2m,3m]$ contain only short edges and are not in scale.
Edges in $G$ are colored blue and nonedges red.
The union of short blue edges contained in an interval between successive
pivots is called blue island.
To each rational number $x<0$, which is not an integer, there
corresponds a graph $G$ of the form shown in the figure.
See Remark \ref{remstu} below for a concrete example.
}
\label{lalala}
\end{figure}
\\
As per our convention,
each edge in $\edges(G) = \edges_1 \cup \edges_\piv \cup \edges_\hop
\cup \edges_\hyper$ is coloured blue and has weight $1$.
All nonedges are coloured red and are given weight $x=-\ell+s/t$.
If $\pi$ is a path between two vertices its weight equals the number of
blue edges of $\pi$ plus $x$ times the number of its red edges.
We next define some special paths that we call {\emph blue islands}:
\[
I_j := \begin{cases}
(a_{j-1}+2, a_{j-1}+3, \ldots, a_j-1), & \text{ if $1 < j \le t$}
\\
(a_0+1, a_0+2, \ldots, a_1-1), & \text{ if $j=1$}
\end{cases}.
\]
Each $I_j$ is a graph whose edges belong to $G$ and so,
by convention, are all blue.
Note that 
\begin{equation}
\label{Ilength}
\text{$I_j$ has length $\ell+v_j-1$ for all $1 \le j \le t$.}
\end{equation}
A blue island may be empty: this happens
if and only if $\ell=1$ and $v_j=0$.
Hence if $x$ is a negative rational number with $|x|>1$ then
all blue islands are nonempty.
It is important to observe that the way an ``interior'' island $I_j$, $2 \le j \le t-1$,
sits within the interval $[a_{j-1}, a_j]$,
is different that the way that a ``boundary'' island $I_1$ or $I_t$ does.
For $j=2, \ldots, t-1$, there are two red short edges 
preceding $I_j$ and one short red edge succeeding it within $[a_{j-1}, a_j]$
The first island $I_1$ has one short short red edge before and one after it in $[a_0,a_1]$.
The last island $I_t$ has {\black two short red edges} before and one blue short edge {\black (from $\edges_\hop$)}
after it in $[a_{t-1},a_t]$.

We next show that $G$ is in $\fH_n$. We check the conditions
of Definition \ref{hndef}.

{\bf (a)} We show that for each $k \in [1,n]$ there is a blue path from $0$ to $k$
(abbreviate this as $0 \leadsto k$).
(i) If $1 \le k \le a_0$ then $0 \leadsto k$ via short edges.
(ii) If $k$ is a pivot then  $0 \leadsto a_0$ via short edges and $a_0 \leadsto a_j$ 
via edges in $\edges_\piv$.
(iii) If $a_j+2 \le k \le a_{j+1}-1$, $1 \le j \le t-1$,
then there is an edge in $\edges_\hyper$ from
$0$ to $a_j+2$ and then $a_j+2 \leadsto k$ via short edges.
The case $j=0$ is similar.
(iv) If $k=a_j+1$, $1\le j \le t-1$,
then there is an edge $(a_j-1, a_j+1)$ and then $0 \leadsto a_j-1$
by (iii). If $k=a_0+1$ then $(0, a_0+1) \in \edges_\hyper$.

{\bf (b)} 
We can similarly 
show that for each $k \in [0,n-1]$ there is a blue path from $k$ to $n$.
(i) If $a_t \le k \le n-1$ then $k \leadsto n$ via short edges only,
(ii) If $k=a_j$ then $a_j \leadsto n$ via edges in $\edges_\piv$ and $\edges_\short$.
(iii) If $k=a_j+1$, $1 \le j \le t-1$ then $(a_j+1, n) \in \edges_\hyper$.
(iv) If $a_j+2 \le k \le a_{j+1}-1$, $1 \le j \le t-1$,
then $k \leadsto a_{j+1}-1$ via short edges,
$(a_{j+1}-1, a_{j+1}+1) \in \edges_\hop$ and $a_{j+1}+1 \leadsto n$ by (iii).
For $j=0$, we have that $a_0+1 \leadsto a_1-1$ via short edges,
$(a_1-1, a_1+1) \in \edges_\hop$ and reduce to (iii) again.

{\bf (c)}
We next show that for any $k \in [1, n-1]$ there is an $i$ such that
$k$ and $i$ are not connected via a blue path.
(i) If $1 \le k \le a_0$ then take $i=a_0+1$.
(ii) If $a_0+1 \le k \le a_1-1$ then take $i=a_1$.
(iii) If $k = a_j$, $1 \le j \le t-1$, then take $i=a_j+1$
(iv) If $k = a_j+1$, $1 \le j \le t-1$, then take $i=a_j$
(v) If $a_j+2 \le k \le a_{j+1}-1$, $1\le j \le t-2$ then take $i=a_{j+1}$.
(vi) If $a_{t-1}+2 \le k \le n$ then take $i=a_{t-1}+1$.

The arguments of {\bf (a)}, {\bf (b)}, {\bf (c)} above show that $G \in \fH_{n}$.

To simplify notation in what follows, we simply write 
$w(\pi)$, $N(\pi)$, $\overline N(\pi)$
for the weight, number of blue edges, number of red edges, respectively,
for any path $\pi$ with arbitrary endpoints. Thus, $w(\pi) = N(\pi)
+ x \overline N(\pi)$.

We next show that the weight of each path $\pi$ from $0$ to $n$ is at most
$2m+t$:
\begin{equation}
\label{finalmax}
w(\pi) \le 2m+t, \quad \text{for all } \pi \in  \Pi_{0,n}.
\end{equation}

We observe that this weight is achievable by the path
\begin{equation}
\label{pi1}
\pi_1 =(0,1,\ldots,a_0-1,a_0,a_1,\ldots,a_t,a_{t}+1,\ldots,n),
\end{equation}
consisting of the first $m$ short edges of $[0,m]$,
the $t$ edges of $\edges_\piv$ and the last
$m$ short edges of $[m,2m]$.
Since it has no red edges its weight equals its length, that is, $2m+t$.

We first show that \eqref{finalmax} holds if $\pi$ contains an edge from
$\edges_\hyper$.
Suppose that $\pi \in \Pi_{0,n}$ has an edge in $\edges_\hyper$ of the form $(0,k)$.
Then $k > a_0=m$. Since the restriction of $\pi$ on $[k,n]$ has weight at most $n-k$,
the weight of $\pi$ is at most $n-k+1 < n-m+1 = 2m+1<2m+t$ since $t>1$.
Similarly, if $\pi$ uses a an edge from $\edges_\hyper$ of the form $(k,n)$,
we have $k< 2m$ and therefore the weight of $\pi$ is at most $k+1 < 2m+t$.

We next show that if $\pi \in \Pi_{0,n}$ has no edges in $\edges_\hyper$ 
then there is a path $\pi' \in \Pi_{0,n}$ containing all short blue edges
in $[0,m]$ and all short blue edges in $[2m,3m]$ such that
$w(\pi') \ge w(\pi)$.
Let $0, u_1, u_2, \ldots, u_r$ be the vertices of $\pi$ on the segment $[0,m]$,
listed in increasing order, and let $u$ be the vertex of $\pi$ succeeding $u_r$. {\black Assume $r<m$. Then at least one of the edges $(0,u_1), (u_1,u_2), \ldots, (u_r,u)$ must be red.}
Replace the subpath $(0,u_1, \ldots, u_r,  u)$,
that  has weight {\black at most} $r+x$, 
by the subpath $(0,1,\ldots, m, u)$, that has weight $m+x
> r+x$ {\black and obtain a contradiction with the maximality of $\pi$.}
So it suffices to show the statement \eqref{finalmax} for paths
that contain all vertices $[0,m]$ and, similarly, all vertices in $[2m, 3m]$.
Such paths have weight $2m$ plus the weight of their restriction on 
$[m,2m]=[a_0, a_t]$.

Therefore, it is enough to show \eqref{finalmax} for all $\pi \in \Pi_{0,n}$
that contain all short blue edges
in $[0,m]$ and all short blue edges in $[2m,3m]$.
Such a $\pi$, necessarily, has no edges in $\edges_\hyper$.
Moreover, $w(\pi) = 2m + w(\pi')$ where $\pi' \in \Pi_{a_0,a_t}$.
Hence \eqref{finalmax} will be proved once we prove that
\begin{equation}
\label{action}
w(\pi) \le t, \quad \text{for all } \pi \in \Pi_{a_0,a_t}.
\end{equation}
To show this, we need to consider the edges of 
any $\pi\in\Pi_{a_0,a_t}$ that are in $\edges_\piv$,
each of which contributes $1$ to the weight of $\pi$,
as well as the restiction of $\pi$ on the intervals between the end of a
pivot edge and the beginning of the next one.
Suppose that the following claim is true.
\begin{claim}
\label{claim1}
If, for all ${\black 0} \le i < j \le t$ and any path $\pi \in \Pi_{a_i, a_j}$
that contains no edges from $\edges_\piv$,
we have $w(\pi) \le j-i$ then \eqref{action} holds.
\end{claim}
To see how \eqref{action} follows from this claim,
let $e_j:=(a_{j-1},a_j)$, $1 \le j \le t$, be a labelling of the elements
of $\edges_\piv$, and
consider a path $\pi \in \Pi_{a_0,a_t}$. 
Let $e_{j_1}, \ldots, e_{j_r}$ be the edges of $\pi$ that are in $\edges_\piv$, 
where $1 \le j_1 < j_2 < \cdots < j_r \le t$.
Consider the restrictions $\pi^{(j_1)}$, $\pi^{(j_2)}, \ldots, 
\pi^{(j_r)}, \pi^{(j_{r+1})}$ of
$\pi$ on the sets of vertices $[a_0, a_{j_1-1}]$, $[a_{j_1}, a_{j_2-1}],
\ldots, [a_{j_r}, a_{t}]$, respectively
(noting that some of these restrictions may be trivial).
Since none of these restrictions contain edges from $\edges_\piv$, 
we have, by Claim \ref{claim1},
\[
w(\pi^{(j_1)}) \le j_1-1,
\quad w(\pi^{(j_k)}) \le (j_k-1){\black -j_{k-1}}, k=2,\ldots,r, 
\quad w(\pi^{(j_{r+1})}) \le t-j_r.
\]
Hence 
\[
w(\pi) \le \sum_{k=1}^{r+1} w(\pi^{(j_k)}) + r,
\]
where the last $r$ is the total weight of all edges of $\pi$ from $\edges_\piv$;
see Figure \ref{lalala}.
Thus
\[
w(\pi) \le j_1-1 + \sum_{k=2}^{r-1} ((j_k-1)-j_{k-1}) + (t-j_r)+r
= j_1 + \sum_{k=2}^r (j_k-j_{k-1}) + (t-j_r) = t.
\]

It remains to prove Claim \ref{claim1}, and this will be done in a few steps.
First consider $\pi \in \Pi_{a_i,a_j}$ that touches every blue island between $a_i$ and 
$a_j$, meaning that for each $i < k \le j$, $\pi$ has a vertex from $I_k$.
Such a $\pi$, necessarily, has no edges from $\edges_\piv$.

\begin{claim}
\label{claim2}
For all $0 \le i < j \le t$,
and all $\pi \in \Pi_{a_i,a_j}$ that touch each blue island between $a_i$ and $a_j$,
we have $w(\pi) \le j-i$.
\end{claim}

\proof To see this, let $\pi$ be a path as in the statement of the claim.
The number $N(\pi)$ of blue edges of $\pi$ is at most the number of all
blue edges between $a_i$ and $a_j$. 
This  is at most the sum $\sum_{i<k \le j} |I_k|$ of the lengths
of all blue islands between $a_i$ and $a_j$ plus the number of edges from $\edges_\hop$
between $a_i$ and $a_j$;
if $j < t$ there are $j-i-1$ such edges;
if $j=t$ there are $t-i$ such edges.
We consider the two cases separately.
First, assume $j<t$. Then, using \eqref{Ilength},
\begin{multline*}
N(\pi) \le \sum_{i<k \le j} |I_k| + (j-i-1)
= \sum_{i<k \le j} (|I_k|+1) -1 
\\ = \sum_{i<k \le j} (\ell+v_k) -1
= (j-i) \ell + \sum_{i<k \le j} v_k -1.
\end{multline*}
Since $v=(v_1,\ldots, v_t)$ is $(t,t-s)$-balanced, we have, by \eqref{UP},
\[
\sum_{i<k \le j} v_k -1  \le \bigg\lceil (j-i) \frac{t-s}{t}  \bigg\rceil -1
\le (j-i) \frac{t-s}{t}
\]
and so
\begin{equation}
\label{bothcases}
N(\pi) \le (j-i) (\ell + 1 - s/t) = (j-i)(1-x).
\end{equation}
If $j=t$ we have
\begin{multline*}
N(\pi) \le \sum_{i<k \le t} |I_k| + (t-i)
= \sum_{i<k \le t} (|I_k|+1) 
= \sum_{i<k \le t} (\ell+v_k)
= (t-i) \ell + \sum_{i<k \le t} v_k.
\end{multline*}
But, from \eqref{UP},
\[
\sum_{i<k \le t} v_k = \bigg\lfloor (t-i) \frac{t-s}{t}  \bigg\rfloor \le (t-i) \frac{t-s}{t},
\]
and hence \eqref{bothcases} holds for all $j$, including the $j=t$ case.
Since $\pi$ touches every $I_k$, $i<k\le j$, the edge with endpoint the  first
common vertex of $\pi$ and $I_k$ is red, by construction.
Hence the number of red edges of $\pi$ is at least $j-i$.
Recalling that $x<0$, we have
\[
w(\pi) =N(\pi) + x \overline N(\pi) 
\le (j-i) (1-x) + x(j-i) = j-i,
\] 
proving the claim.
\qed

\medskip
We next need to see what happens when the premise of Claim \ref{claim2} fails, that is,
when a path $\pi \in \Pi_{a_i,a_j}$ avoids some island between $a_i$ and $a_j$.
We first consider the case where $\pi$ avoids an $I_k$ other than the last island $I_j$.

\begin{claim}
\label{claim3}
Let ${\black 0} \le i < j \le t$ and $\pi \in \Pi_{a_i,a_j}$ that contains no edges from $\edges_\piv$
and  avoids some island $I_k \neq I_j$.
Then there is another path $\pi'$ that includes $I_k$
and $w(\pi') > w(\pi)$.
\end{claim}

\proof  Consider a path $\pi$  from $a_i$ to $a_j$ with no edges from $\edges_\piv$,
avoiding some $I_k \neq I_j$.
Then $\pi$ contains an edge $(a,b)$ such that
$a$ is strictly smaller than the minimum vertex of $I_k$
and $b$ strictly larger than the maximum vertex of $I_k$.
Since $\pi$ has no edges from $\edges_\piv$, $(a,b)$ is red.
\\
-- If $b=a_k$, then, since $k \neq j$, there is a vertex $c$ of $\pi$
such that $(b,c)$ is an edge of $\pi$.
Necessarily, $(b,c)$ is red.
Consider the subpath $\sigma=(a,b,c)$ of $\pi$ that has weight $w(\sigma)=2x$.
Replace $\sigma$ by the path $\sigma'$
with vertices $a$, all the vertices of $I_k$, {\black $a_k+1$} and $c$.
We estimate the weight $w(\sigma') = N(\sigma') + x \overline N(\sigma')$ from below.
It has $N(\sigma') = |I_k|+1$ blue edges (the edges of $I_k$ and the blue edge
$(a_k-1, a_k+1) \in \edges_\hop$)
and $\overline N(\sigma') \le  2$ red edges.
This is because the first edge of $\sigma'$ is red whereas the last edge,
$(a_k+1,c)$, if it exists (that is, if $c \neq a_k+1$), is also red.
Hence $w(\sigma') \ge |I_k|+1 + 2x = \ell+v_k+ 2x > 2x = w(\sigma)$. 
Thus, replacing $\sigma$ in $\pi$ by $\sigma'$ we increase the weight.
\\
-- If $b > a_k$, replace $\sigma=(a,b)$  by the path $\sigma'$
with vertices $a$, all the vertices of $I_k$, $a_{k+1}$ and $b$.
By exactly the same estimation as in the previous case, 
$w(\sigma') \ge |I_k|+1 + 2x  = \ell+v_k+ 2x \ge \ell + 2x$.
Now, since $x > -\ell$, we have $\ell+2x > x$, so $w(\sigma') > w(\sigma)$.
Again then, replacing $\sigma$ in $\pi$ by $\sigma'$ we increase the weight.
\qed

\begin{claim}
\label{claim4}
Let ${\red 0} \le i < j \le t$ and $\pi \in \Pi_{a_i,a_j}$ that contains no  edges from $\edges_\piv$
and  includes each of the $I_{i+1}, \ldots, I_{j-1}$.
Then $w(\pi) \le j-i$.
\end{claim}

\proof  
Fix a $\pi$ from $a_i$ to $a_j$ with no edges from $\edges_\piv$ that
includes each of the $I_{i+1}, \ldots, I_{j-1}$.

{\black First assume $j\ge 2$.
Note that} either  $\pi$ touches $I_j$ or not.
If it does, use Claim \ref{claim2} to conclude that $w(\pi) \le j-i$.
If it does not, let $a$ be the vertex preceding $a_j$ in $\pi$.
Since $\pi$ does not touch $I_j$ we have that $a \le a_{j-1}+{\black 1}$, provided that
$j \neq 1$.
Since $\pi$ is not allowed to have an edge in $\edges_\piv$,
we must have $a \neq a_{j-1}$. 
{\black If $j=i+1$ then clearly $a=a_{j-1}+1$; otherwise, since} $\pi$ includes $I_{j-1}$, it follows that $a \ge a_{j-1}-1$.
Hence, {\black in any case,} the only possible values of $a$
are $a_{j-1}\pm 1$. The weight of the edge $(a,b)$ is $x$.
\\
-- If $a=a_{j-1}- 1$, let $\tau$ be the subpath of $\pi$ from $a_i$ to $a$.
Hence $w(\pi)=w(\tau) + x$.
Now let $\pi'$ be the path consisting of $\tau$ and followed by
the edges $(a,a_{j-1})$ and $(a_{j-1}, a_j)$;
the first is red, the second blue as it is in $\edges_\piv$.
Hence $w(\pi') = w(\tau)+x+1 > w(\pi)$.
But the path {\black $\tau$} together with the edge $(a,a_{j-1})$ is a path
from $a_i$ to $a_{j-1}$ containing all blue islands between $a_i$ and $a_{j-1}$ and,
by Claim \ref{claim2}, has weight at most $(j-1)-i$.
Hence $w(\pi) < w(\pi') \le (j-1)-i+1 = j-i$.
\\
-- If $a=a_{j-1}+1$, let $c$ be the vertex of $\pi$ preceding $a$
and let $\tau$ be the subpath of $\pi$ from $a_i$ to $c$. {\black If $c=a_{i-1}$ then simply replace two red edges $(c,a)$ and $(a,a_j)$ by a single blue edge $(c,a_j)$. Otherwise, the}
weight $w(c,a)$ of the last edge $(c,a)$ of $\tau$ is either $1$ if
$z=a_{j-1}-1$ or $x$, otherwise.
Hence $w(\pi) \le  w(\tau) + 1+x$.
Consider now the path $\pi'$ consisting of $\tau$ and followed by
the red edge $(c, a_{j-1})$ and the blue edge $(a_{j-1},a_j) \in {\black \edges_\piv}$.
Hence $w(\pi') = w({\black \tau}) + 1 + x \ge w(\pi)$.
On the other hand, arguing as above, $w(\pi') \le j-i$.
Hence, again $w(\pi)\le j-i$.
\\
It remains to consider the case $j=1$. Then $i=0$ and so we must consider
a path $\pi$ from $a_0$ to $a_1$.
If $\pi$ touches $I_1$ then, by Claim \ref{claim2}, it has weight at most $1$.
If it does not, then, necessarily, $\pi$ consists of the single edge $(a_0,a_1)$,
which is impossible since  $(a_0,a_1) \in \edges_\piv$.
\qed

\proof[Proof of claim \ref{claim1}]
Consider a path $\pi \in \Pi_{a_i,a_j}$
with no edges from $\edges_\piv$.
If this path touches each blue island between $a_i$ and $a_j$ then, by Claim \ref{claim2},
it has weight at most $j-i$.
If it avoids some of the $I_k$, $i <k < j$, then, by Claim \ref{claim3}, there
is another path $\pi'$ that includes all of the $I_k$, $i <k < j$,
such that $w(\pi) \le w(\pi')$.
But $\pi'$, by Claim \ref{claim4}, has weight at most $j-i$.
\qed

\medskip
To finish the proof of Proposition \ref{-qcrit}, it remains to show that there are
two different maximal paths from $0$ to $n$,
that is, both with weight $2m+t$,
but with different number of edges. The first one is $\pi_1$ defined by \eqref{pi1}.
The second one is
\[
\pi_2 :=(0,1,\ldots, a_1-1,a_1+1,\ldots,a_2-1,a_2+1, 
\ldots, a_{t-1}-1, a_{t-1}+1,\ldots,n).
\]
It contains $[0,m] \cup [m,2m]$ (as it should).
It also contains all blue islands, all edges in $\edges_\hop$
and has a nonzero number of red edges.
To find its weight we count the number of edges of each color.
Blue island $I_j$ has length $\ell-1+v_j$ for all $1 \le j \le t$.
In addition to these, $\pi_2$ contains all {\black $t$} edges in $\edges_\hop$ 
and the blue edge $(a_t-1, a_t)$; hence, in total,
$\pi_2$ has $2m + \sum_{j=1}^t (\ell-1+v_j)+t
= 2m+\ell t + \sum_{j=1}^t v_j = 2m+\ell t + (t-s)$ blue edges.
On the other hand, 
it  contains exactly $t$ red edges, the edges
$(a_0, a_0+1)$, $(a_1+1, a+1+2), \ldots, (a_{t-1}+1, a_{t-1}+2)$.
Hence its weight is $xt + (2m+\ell t + t-s) = 2m+t$.
This concludes the proof of the criticality of each rational $x< 0$ which
is not an integer.
\end{proof}

\begin{remark}[Sturm graph]
\label{remstu}
In the proof above, we constructed, for each negative rational number $x$
a graph $G \equiv G(x)$ that depends only on $x$, in fact, on its representation
as $x=-\ell + (s/t)$, where $-\ell=\lfloor x \rfloor$ and $s, t$ coprime
positive integers. (Actually, the coprimality of $s$ and $t$ is not important in the 
proof above.)
The graph is on $n=  3 t(\ell+3)- 3(s+1)$ vertices
and has edges distributed according to the unique binary word of length
$t$ that is $(t,t-s)$-balanced, as obtained by Lemma \ref{sturmlem},
a corollary to Sturm's lemma. 
Hence the map $x \mapsto  G(x)$ is a bijection.
For lack of any terminology, we refer to $G(x)$ as a {\bf\em Sturm graph}.
This remark might be of independent interest in future
research on properties of weighted Barak-Erd\H{o}s graphs.
\end{remark}

\begin{remark}[An example]
\label{remexa}
As an example, we take $x=-11/7=-2+(3/7)$.
Then $\ell=2, s=3, t=7$, $m=31$, $n=93$.
Figure \ref{lalala} above is actually drawn for this example.
We can check that $(v_1, \ldots, v_7) = (1,1,0,1,0,1,0)$ is the  $(7,4)$-balanced
sequence.
Hence the pivots $a_0, \dots, a_7$ are at distances
$a_1-a_0= \ell+1+v_1=4$, 
$a_2-a_1 = \ell+2+v_2 = 5$,
$a_3-a_2=4$, 
$a_4-a_3=5$,
$a_5-a_4=4$, 
$a_6-a_5=5$, 
$a_7-a_6=4$.
In the proof of \eqref{-qcrit}, we showed that the maximum weight of all paths from 
$a_0$ to $a_t$ is $t$ and that there is a maximal path from 
$a_0$ to $a_t$ that includes all blue islands. 
However, if $i\neq 0$ or $j \neq t$, even though the maximum weight of all paths from 
$a_i$ to $a_j$ is $j-i$, it is not necessarily true thay there is a maximal path
from $a_i$ to $a_j$ that includes all blue islands in between.
Indeed, let $i=1$ and $j=3$ and consider the path
\[
\pi = (a_1,a_1+2, a_1+3, a_2-1, a_2+1, a_2+2, a_3-1, a_3).
\]
This path contains the blue islands $I_2$ and $I_3$
and has weight $4+3x=-5/7$.
Consider also the path
\[
\pi':=(a_1,a_1+2, a_1+3, a_2-1, a_2+1, a_3)
\]
that avoids $I_3$. See Figure \ref{exfig}.
It has weight $3+2x=-1/7$, larger than the weight of $\pi$.
\begin{figure}[h] 
\includegraphics[width=8cm,]{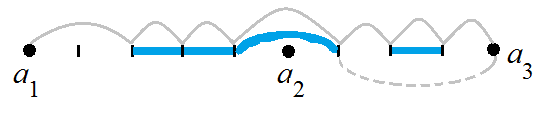}
\captionof{figure}
{Consider two paths, $\pi, \pi'$, from $a_1$ to $a_3$.
The first one, $\pi$, is represented by the solid gray line.
The second one, $\pi'$, coincides with $\pi$ up to vertex $a_2+1$
and then uses the dotted edge.
We have $w(\pi) < w(\pi')$ even though $\pi'$ does not use the second
blue island.
}
\label{exfig}
\end{figure}
\end{remark}

\begin{proof}[Proof of Theorem \ref{thmcrit} and Theorem \ref{thmmain}]
By Lemma \ref{0crit} and Propositions \ref{-ncrit} and \ref{-qcrit} every
rational number that is less than or equal to $0$ is critical.

By Lemma \ref{1noncrit} and Proposition \ref{+crit},
 every number of the form $k$ or $1/k$,
where $k$ is a positive integer other than $1$, is critical.

Theorem \ref{thmdiffcrit} says that the set of critical points
is precisely the set of points where $C_p$ is not differentiable.

By Lemma \ref{irrdiff}, if $x \ge 0$ and irrational then $C_p$ is not differentiable
at $x$ and, by Theorem \ref{thmdiffcrit}, $x$ is not critical.

By Proposition \ref{+noncrit01} every number in $[0,1)$, 
except the reciprocals of positive integers, is not critical.
Using the scaling property (Proposition \ref{scaling1}) we have that
every number in $(1,\infty)$ that is not an integer is not critical.

Hence the set of critical points is precisely the one described in the statement of
Theorem \ref{thmcrit}, and so this theorem has been proved.

By Theorem \ref{thmdiffcrit}, the set of critical points coincides with
the set of points at which $C_p$ is not differentiable, and hence Theorem
\ref{thmmain} has been proved as well.
\end{proof}


\section{Additional remarks}
The problem we studied is a special instance of a more general one where
the edge weight is a random weight whose distribution $F=F_{\theta}$
depends on a parameter $\theta$.
The rate of growth of the heaviest path (the analog of $C_p(x)$; see Theorem \ref{thmx}) exists. Let it be denoted by $C(F_{\theta})$.
We can immediately translate the results of this paper to cover some
instances of this problem exhibiting the behaviour of the function
$\theta \mapsto C(F_{\theta})$.

We may continue to fix $p$ and to consider $\theta=x$.
Suppose that $F_x= (1-p) \delta_x +p Q$ where $Q$ is a probability
measure on $(0,\infty)$. Then  we expect that the behavior of
$C( (1-p) \delta_x +p Q)$ as a function of $p$ may be derived by using 
the techniques developed in \cite{MR16,MR19}, provided that
$\int y^2 Q(dy) < \infty$.
As a function of $x$, $C( (1-p) \delta_x +p Q)$ is continuous and convex.
Let
\[
K:= \{x < 0:\, -x \in \Q\} \cup \{x >0:\, x \in \N\setminus\{1\}\}
\cup \{x >0:\, 1/x \in \N\setminus\{1\}\}
\]
(this is the set that appeared in Theorem \ref{thmcrit}).
Assume that $Q$ is supported on $(0,\infty)$.
Then the set at which $C( (1-p) \delta_x +p Q)$ is not
a differentiable function of $x$ is the set
\[
\{x \in \R:\, \exists y \text{ such that } Q\{y\}>0 \text{ and } x/y \in K\}.
\]
In particular, if $Q$ has no atoms then 
$C( (1-p) \delta_x +p Q)$ is differentiable at all $x \in \R$.
To prove this claim, it suffices to consider the case $Q=\delta_y$, $y>0$.
Then, by an obvious scaling, 
\[
C( (1-p) \delta_x +p \delta_y) = y C_p(x/y), 
\]
where $C_p(\cdot)$ is as in Theorem \ref{thmx}.
The claim then follows from Theorems \ref{thmdiffcrit} and \ref{thmcrit}.
If $Q$ has positive mass on the negative real numbers then the behaviour is more
involved and may lead to studying completely different situations. For example, if $Q$ is supported on $(-\infty,0)$, 
then, switching the signs, the maximization problem is transformed
into a minimization one. To analyse the latter, a different technique is required. One may introduce further weights on
vertices like in \cite{FK18} and/or like in \cite{DFK12} and
analyse similar (non)differentiability questions. 
These extensions may lead to a new interesting direction of research with
many open problems. 
Another direction would be to study further analytical problems of $C(F_x)$, like existence of further derivatives in the case where
$Q$ is a continuous distribution. Another direction would be to
study properties of $C(F_{\theta})$ as a function of  two-dimensional parameter $\theta = (x,p)$.


\section*{Acknowledgements}
The research of TK was partially supported by 
the CNRS PRC collaborative grant CNRS-193-382.
The research of SF was partially supported by
the Akademgorodok  Mathematical Centre under  agreement no.\ 075-15-2019-1675
with the Ministry of  Science and Higher Education.
The research of AP was supported by 
the Sobolev Institute of Mathematics contract no.\ 0314-2019-0014.

The authors would like to thank an anonymous referee who read the manuscript
in great detail, made us aware of references \cite{KRS} and \cite{SZ},
and  suggested various improvements.
The referee's remarks helped us to largely improve the presentation of the paper, 
and, in particular, to provide a short proof for Lemma \ref{sturmlem} and add all
details of Proposition \ref{-qcrit}.

\vspace*{4cm}
\noindent
\underline{Sergey Foss},
School of Mathematical Sciences, Heriot-Watt University, Edinburgh;
\href{mailto:s.foss@hw.ac.uk}{s.foss@hw.ac.uk}
\\[4mm]
\noindent
\underline{Takis Konstantopoulos},
Department of Mathematical Sciences,
The University of Liverpool;
\href{mailto:takiskonst@gmail.com}{takiskonst@gmail.com}
\\[4mm]
\noindent
\underline{Artem Pyatkin}, Sobolev Institute of Mathematics, Novosibirsk;
\href{mailto:artempyatkin@gmail.com}{artempyatkin@gmail.com}

\end{document}